\documentclass[twocolumn,abstract,10pt]{scrartcl}

\usepackage[letterpaper, left=1.2cm, right=1.2cm, top=1.2cm, bottom=1.2cm]{geometry}
\usepackage[english]{babel}
\usepackage{amsmath}
\usepackage{amssymb}
\usepackage{amsthm}
\usepackage{enumitem}
\usepackage[colorlinks = true, linkcolor = blue, citecolor = blue]{hyperref}
\usepackage[nameinlink]{cleveref}
\usepackage{graphicx}

\setcapindent{0pt}
\setkomafont{caption}{\footnotesize}
\setkomafont{captionlabel}{\bfseries}

\newtheorem{assumption}{Assumption}
\newtheorem{corollary}{Corollary}
\newtheorem{definition}{Definition}
\newtheorem{proposition}{Proposition}
\newtheorem{remark}{Remark}
\newtheorem{theorem}{Theorem}

\crefrangelabelformat{subequation}{(#3#1#4--#5\crefstripprefix{#1}{#2}#6)}
\crefname{equation}{}{}
\crefname{assumption}{Assumption}{Assumptions}
\crefname{corollary}{Corollary}{Corollaries}
\crefname{definition}{Definition}{Definitions}
\crefname{proposition}{Proposition}{Propositions}
\crefname{remark}{Remark}{Remarks}
\crefname{theorem}{Theorem}{Theorems}

\addto\captionsenglish{}

\title{$\quad$Extremum Seeking Control for a Class of $\quad$ Mechanical Systems\footnote{This research was supported by the German Research Foundation DFG, project number DA~767/13-1.}}
\author{Raik Suttner\footnote{University of Wuerzburg, Wuerzburg, Germany (email: raik.suttner@ mathematik.uni-wuerzburg.de)}}
\date{}

\begin{document}
\maketitle

$\vphantom{m}$ \vspace{-2cm} $\vphantom{m}$

\begin{abstract}
We present a novel extremum seeking method for affine connection mechanical control systems. The proposed control law involves periodic perturbation signals with sufficiently large amplitudes and frequencies. A suitable averaging analysis reveals that the solutions of the closed-loop system converge locally uniformly to the solutions of an averaged system in the large-amplitude high-frequency limit. This in turn leads to the effect that stability properties of the averaged system carry over to the approximating closed-loop system. Descent directions of the objective function are given by symmetric products of vector fields in the averaged system. Under suitable assumptions, we prove that minimum points of the objective function are asymptotically stable for the averaged system and therefore practically asymptotically stable for the closed-loop system. We illustrate our results by examples and numerical simulations.
\end{abstract}

%---------------------------------------------------------------------------------------------------------------------------------------------------------
%---------------------------------------------------------------------------------------------------------------------------------------------------------
%                                                                     Section 1
%---------------------------------------------------------------------------------------------------------------------------------------------------------
%---------------------------------------------------------------------------------------------------------------------------------------------------------

\section{Introduction}
In this paper, we investigate the problem of asymptotically stabilizing a mechanical system about configurations in which a real-valued objective (or output) function attains an extreme value. The only information about the objective function is given by real-time measurements of its values in the current configurations of the system. On the other hand, measurements of the system state are not assumed to be available. Research on this type of optimization problem is motivated by various applications, such as spacecraft attitude optimization~\cite{Walsh2017}, optimal alignment of solar panels~\cite{Solis2019}, and environmental source seeking~\cite{Matveev2014}.

Classical extremum seeking schemes, like those in~\cite{Krstic2000,Tan2006}, rely on the assumption that the underlying control system is asymptotically stable for arbitrary constants inputs. This assumption is certainly valid for some extremum seeking problems (see e.g.~\cite{AriyurBook,ZhangBook}), but, in general, one cannot expect that a control system displays a stable behavior in the presence of constant forces or torques. It is therefore desirable to develop methods that can be applied to potentially unstable control systems. To the best of our knowledge, there is no general extremum seeking method so far that can be successfully applied to a large class of open-loop unstable mechanical systems. The intention of the present paper is to take a first step into this direction.

In recent years, differential geometric concepts have become increasingly popular for the design of extremum seeking control laws~\cite{Duerr2013,Scheinker20132,Labar2019,Michalowsky2020}. These methods employ periodic dither signals with sufficiently large amplitudes and frequencies to induce an approximation of Lie brackets. A suitable design of the feedback loop leads to the effect that the Lie brackets point into descent directions of the objective function. Moreover, the employed large-amplitude high-frequency perturbation signals have the ability to overpower potentially unstable dynamics. This approach works very well as long as the control system is a first-order kinematic system. In general, however, one cannot expect that a direct application of these methods to higher-order control systems, like a second-order mechanical system, will lead to an extremum seeking closed-loop system. Extensions of the first-order kinematic Lie bracket approach from~\cite{Duerr2013} to a class of higher-order systems can be found in~\cite{Michalowsky2014,Michalowsky2015}. However, those extensions are restricted to integrator chains and do not cover the class of mechanical system that we consider in the present paper. Moreover, the approach in~\cite{Michalowsky2014,Michalowsky2015} leads to oscillatory velocities with large amplitudes, which is undesired in real-world applications. Here, we propose a method that guarantees uniformly bounded velocities in the closed-loop system.

Most of the known extremum seeking methods for potentially unstable second-order control systems are restricted to the double-integrator point model. This very simple mechanical control system frequently appears in the literature on source seeking; e.g., in~\cite{Zhang20072,Stankovic2009,Scheinker20183,Fu2021}. For second-order systems with complex dynamics, the problem of extremum seeking becomes significantly more difficult.

As indicated earlier, the Lie bracket approach for first-order kinematic systems in~\cite{Duerr2013} only allows limited extensions to second-order mechanical system. In the present paper we propose a different strategy, which is applicable to a much larger class of mechanical systems. Instead of approximating Lie brackets of pairs of vector field, we approximate so-called symmetric products of vector fields. Symmetric products occur naturally in the analysis of mechanical systems~\cite{Crouch1981,Lewis1998,Bullo2002,BulloBook}. One can show that symmetric products originate from iterated Lie brackets of vector fields on the tangent bundle of the configuration manifold~\cite{Lewis19972,BulloBook}. Here, we use approximations of symmetric products to obtain information about descent directions of the objective function. Applications of this strategy to double-integrator points and unicycles can be found in~\cite{Suttner20192,Suttner20202,Suttner20193}. In this paper, we investigate this new symmetric product approach to extremum seeking control for a more general class of mechanical systems and indicate further applications.

The paper is organized as follows. We summarize some of the frequently used notation in \Cref{sec:02}. In \Cref{sec:03}, we introduce suitable notions of (practical) asymptotic stability for a parameter-dependent system and its associated averaged system. A general extremum seeking problem for a class of mechanical systems is formulated in \Cref{sec:04}. This section also contains the proposed control law and the general stability results. We apply the method to the problems of source seeking and attitude control in \Cref{sec:05}. The paper ends with some conclusions in \Cref{sec:06}.

%---------------------------------------------------------------------------------------------------------------------------------------------------------
%---------------------------------------------------------------------------------------------------------------------------------------------------------
%                                                                     Section 2
%---------------------------------------------------------------------------------------------------------------------------------------------------------
%---------------------------------------------------------------------------------------------------------------------------------------------------------

\section{Notation}\label{sec:02}
We assume that the reader is familiar with the concepts of affine differential geometry. In this paper, we use basically the same notation and terminology as in~\cite{BulloBook}. Since we restrict our considerations to submanifolds of Euclidean space, we can use the following natural identifications of tangent spaces and tangent bundles. Let $Q$ be a submanifold of Euclidean space, say $\mathbb{R}^N$. For every $q\in{Q}$, the \emph{tangent space} to $Q$ at $q$, denoted by $T_qQ$, is treated as a subspace of $\mathbb{R}^N$. The \emph{tangent bundle} of $Q$, denoted by $TQ$, is the disjoint union of all tangent spaces to $Q$; i.e., $TQ$ is the subset of $\mathbb{R}^N\times\mathbb{R}^N$ consisting of all pairs $(q,v)$ with $q\in{Q}$ and $v\in{T_qQ}$. In particular, if $q$ is a differentiable curve in $Q$, then $(q,\dot{q})$ is a curve in $TQ$. A \emph{vector field} $Y$ on $Q$ is considered as a map $Q\to\mathbb{R}^N$ such that $Y(q)\in{T_qQ}$ for every $q\in{Q}$. A \emph{bundle map} $R$ over $Q$ is treated as a map that assigns to each $q\in{Q}$ a linear map $R(q)\colon{T_qQ}\to{T_qQ}$. A \emph{Riemannian metric} $\langle\!\langle\cdot,\cdot\rangle\!\rangle$ on $Q$ is a map that assigns to each $q\in{Q}$ an inner product $\langle\!\langle\cdot,\cdot\rangle\!\rangle_q$ on $T_qQ$ in a smooth way. The induced norm on $T_qQ$ is then denoted by $\|\cdot\|_q$. The word ``smooth'' always means ``of class $C^\infty$.'' For every smooth real-valued function $\psi$ on $Q$ and every vector field $X$ on $Q$, the \emph{Lie derivative} of $\psi$ along $X$ is denoted by $X\psi$. An \emph{affine connection on $Q$} is a map $\nabla$ that assigns to each pair of smooth vector fields $X,Y$ on $Q$ another smooth vector field on $Q$, denoted by $\nabla_{X}Y$, such that $(X,Y)\mapsto\nabla_{X}Y$ is $\mathbb{R}$-linear and such that $\nabla_{\psi{X}}Y=\psi\nabla_{X}Y$ and $\nabla_X(\psi{Y})=\psi\nabla_X{Y}+(X\psi)Y$ for every smooth real-valued function $\psi$ on $Q$. Let $q$ be a smooth curve in $Q$. Note that the second-order time derivative $\ddot{q}(t)\in\mathbb{R}^N$ is, in general, not an element of $T_{q(t)}Q$. However, one can use the defining properties of an affine connection $\nabla$ on $Q$ to introduce a suitable replacement $\nabla_{\dot{q}(t)}\dot{q}(t)$ for $\ddot{q}(t)$ as an element of $T_{q(t)}Q$ (see~\cite{BulloBook}). The resulting map, denoted by $\nabla_{\dot{q}}\dot{q}$, is called the \emph{geometric acceleration} of $q$.

%---------------------------------------------------------------------------------------------------------------------------------------------------------
%---------------------------------------------------------------------------------------------------------------------------------------------------------
%                                                                     Section 3
%---------------------------------------------------------------------------------------------------------------------------------------------------------
%---------------------------------------------------------------------------------------------------------------------------------------------------------

\section{Practical stability}\label{sec:03}
Let $M$ be an embedded submanifold of Euclidean space with Euclidean norm $|\cdot|$. If $K$ is a nonempty and compact subset of $M$, then, for every $x\in{M}$, we use the notation $|x|_K$ for the Euclidean distance of $x$ to $K$. For every $\omega>0$, let $f^\omega$ be a time-dependent vector field on $M$ such that, for every $t_0\in\mathbb{R}$ and every $x_0\in{M}$, the differential equation%
\begin{equation}\label{eq:01}
\dot{x} \ = \ f^\omega(t,x)
\end{equation}%
with initial condition $x(t_0)=x_0$ has a unique maximal solution. The extremum seeking closed-loop system in \Cref{sec:04} will be of the form~\cref{eq:01}, where $\omega$ is a parameter to scale the amplitudes and frequencies of periodic dither signals. For every $\omega>0$ and every $t\in\mathbb{R}$, let $\phi^\omega_{t}\colon{M}\to{M}$ be a bijective map. We use this map to carry out the \emph{change of variables}%
\begin{equation}\label{eq:02}
\tilde{x} \ = \ \phi^{\omega}_t(x).
\end{equation}%
In \Cref{sec:04}, such a change of variables will be applied to the velocities of the closed-loop system.
\begin{definition}[\cite{Duerr2013,Duerr2014}]\label{def:01}
Let $K$ be a nonempty and compact subset of $M$. We say that $K$ is \emph{practically uniformly stable for~\cref{eq:01} in the variables~\cref{eq:02}} if, for every $\varepsilon>0$, there exist $\omega_0,\delta>0$ such that, for every $\omega\geq\omega_0$, every $t_0\in\mathbb{R}$, and every $\tilde{x}_0\in{M}$ with $|\tilde{x}_0|_K\leq\delta$, the maximal solution $x$ of~\cref{eq:01} with initial condition $x(t_0)=(\phi^{\omega}_{t_0})^{-1}(\tilde{x}_0)$ satisfies $|\phi^{\omega}_t(x(t))|_K\leq\varepsilon$ for every $t\geq{t_0}$.
\end{definition}
\begin{definition}[\cite{Duerr2013,Duerr2014}]\label{def:02}
Let $K$ be a nonempty and compact subset of $M$ and let $S$ be a neighborhood of $K$ in $M$. We say that $K$ is \emph{$S$-practically uniformly attractive for~\cref{eq:01} in the variables~\cref{eq:02}} if, for all $r,\varepsilon>0$, there exist $\omega_0,\Delta,R>0$ such that, for every $\omega\geq\omega_0$, every $t_0\in\mathbb{R}$, and every $\tilde{x}_0\in{S}$ with $|\tilde{x}_0|_K\leq{r}$, the maximal solution $x$ of~\cref{eq:01} with initial condition $x(t_0)=(\phi^{\omega}_{t_0})^{-1}(\tilde{x}_0)$ satisfies $|\phi^{\omega}_t(x(t))|_K\leq{R}$ for every $t\geq{t_0}$ and $|\phi^{\omega}_t(x(t))|_K\leq\varepsilon$ for every $t\geq{t_0+\Delta}$.
\end{definition}
\begin{definition}[\cite{Duerr2013,Duerr2014}]\label{def:03}
Let $K$ be a nonempty and compact subset of $M$ and let $S$ be a neighborhood of $K$ in $M$. We say that $K$ is \emph{$S$-practically uniformly asymptotically stable for~\cref{eq:01} in the variables~\cref{eq:02}} if $K$ is practically uniformly stable for~\cref{eq:01} in the variables~\cref{eq:02} and if $K$ is $S$-practically uniformly attractive for~\cref{eq:01} in the variables~\cref{eq:02}.
\end{definition}
\begin{remark}\label{rmk:01}
If $S=M$, then we replace the prefix ``$S$-'' in \Cref{def:02,def:03} by ``semi-globally''. If there exists a time-dependent vector field $\bar{f}$ on $M$ such that $f^\omega=\bar{f}$ for every $\omega>0$, then we omit the word ``practically'' in \Cref{def:01,def:02,def:03}. If $\phi^{\omega}_t(x)=x$ for every $\omega>0$, every $t\in\mathbb{R}$, and every $x\in{M}$, then we omit the phrase ``in the variables~\cref{eq:02}'' in \Cref{def:01,def:02,def:03}. The word ``uniformly'' in \Cref{def:01,def:02,def:03} indicates that the properties therein are uniform with respect to the time parameter.
\end{remark}
Let $\bar{f}$ be a time-dependent vector field on $M$ such that, for every $t_0\in\mathbb{R}$ and every $\bar{x}_0\in{M}$, the differential equation%
\begin{equation}\label{eq:03}
\dot{\bar{x}} \ = \ \bar{f}(t,\bar{x})
\end{equation}%
with initial condition $\bar{x}(t_0)=\bar{x}_0$ has a unique maximal solution. In \Cref{sec:04}, the averaged system of the closed-loop system will be of the form~\cref{eq:03}.
\begin{definition}[\cite{Duerr2013,Duerr2014}]\label{def:04}
We say that \emph{the solutions of~\cref{eq:01} in the variables~\cref{eq:02} approximate the solutions of~\cref{eq:03}} if, for every compact subset $C$ of $M$ and all $\nu,\Delta>0$, there exists $\omega_0>0$ such that, for every $t_0\in\mathbb{R}$ and every $\bar{x}_0\in{C}$, the following implication holds: If the maximal solution $\bar{x}$ of~\cref{eq:03} with initial condition $\bar{x}(t_0)=\bar{x}_0$ satisfies $\bar{x}(t)\in{C}$ for every $t\in[t_0,t_0+\Delta]$, then, for every $\omega\geq\omega_0$, the maximal solution $x$ of~\cref{eq:01} with initial condition $x(t_0)=(\phi^{\omega}_{t_0})^{-1}(\bar{x}_0)$ satisfies $|\phi^{\omega}_t(x(t))-\bar{x}(t)|\leq\nu$ for every $t\in[t_0,t_0+\Delta]$.
\end{definition}
In the situation of \Cref{def:04}, stability properties of~\cref{eq:03} carry over to the approximating system~\cref{eq:01} as follows.
\begin{proposition}[\cite{Duerr2013,Duerr2014}]\label{prop:01}
{\color{red}Suppose that the embedded submanifold $M$ is a topologically closed subset of the ambient Euclidean space.} Let $K$ be a nonempty and compact subset of $M$ and let $S$ be a neighborhood of $K$ in $M$. Assume that $K$ is $S$-uniformly asymptotically stable for~\cref{eq:03} and that the solutions of~\cref{eq:01} in the variables~\cref{eq:02} approximate the solutions of~\cref{eq:03}. Then $K$ is $S$-practically uniformly asymptotically stable for~\cref{eq:01} in the variables~\cref{eq:02}.
\end{proposition}%

%---------------------------------------------------------------------------------------------------------------------------------------------------------
%---------------------------------------------------------------------------------------------------------------------------------------------------------
%                                                                     Section 4
%---------------------------------------------------------------------------------------------------------------------------------------------------------
%---------------------------------------------------------------------------------------------------------------------------------------------------------

\section{Extremum seeking control}\label{sec:04}
Throughout this section, we suppose that
\begin{itemize}
	\item $Q$ is a {\color{red}properly} embedded submanifold of Euclidean space,
	%\item $\mathbb{G}$ is a Riemannian metric on $Q$
	%\item $\nabla$ is the Levi-Civita connection associated with $(Q,\mathbb{G})$,
	\item $\nabla$ is an affine connection on $Q$,
	\item $Y_0,Y_1,\ldots,Y_n$ are smooth vector fields on $Q$,
	\item $R$ is a smooth bundle map over $Q$,
	\item $\psi$ is a smooth real-valued function on $Q$
\end{itemize}
(see~\cite{BulloBook} and \Cref{sec:02} for definitions).

%---------------------------------------------------------------------------------------------------------------------------------------------------------
%                                                                     Section 4.1
%---------------------------------------------------------------------------------------------------------------------------------------------------------

\subsection{Problem statement and control law}\label{sec:04.1}
We consider a multiple-input, single-output mechanical control system of the form%
\begin{align}
\nabla_{\dot{q}}\dot{q} & \ = \ Y_0(q) - R(q)\dot{q} + \sum_{i=1}^n\mathrm{u}^i\,Y_i(q), \label{eq:04} \allowdisplaybreaks \\
\mathrm{y} & \ = \ \psi(q) \label{eq:05}
\end{align}%
on the \emph{configuration manifold} $Q$, where $\mathrm{u}$ is an $n$-component vector of real-valued \emph{input channels} $\mathrm{u}^1,\ldots,\mathrm{u}^n$ and $\mathrm{y}$ is a real-valued \emph{output channel}. The symbol $\nabla_{\dot{q}}\dot{q}$ on the left-hand side of~\cref{eq:04} denotes the \emph{geometric acceleration} of $q$ with respect to $\nabla$ (see~\cite{BulloBook} and \Cref{sec:02}). On the right-hand side of~\cref{eq:04}, the vector field $Y_0$ represents a \emph{drift}, the bundle map $-R$ describes \emph{dissipation}, and $Y_1,\ldots,Y_n$ are referred to as the \emph{control vector fields}. The goal is to asymptotically stabilize~\cref{eq:04} about configurations where the \emph{objective function} $\psi$ attains a minimum value. Only real-time measurements of~\cref{eq:05} can be utilized in a feedback law. Additional information about $\psi$, like its derivative, is not available. Also the current configuration $q$ and the current velocity $\dot{q}$ are treated as unknown quantities.

Now we describe the proposed control strategy. To this end, we introduce suitable functions and parameters. First, we choose a smooth real-valued function $\alpha$ on $\mathbb{R}$ such that the product of $\alpha$ and its derivative $\alpha'$ is positive and strictly increasing. For example,%
\begin{equation}\label{eq:06}
\alpha(z) \ := \ \sqrt{{z}+\log(2\cosh{z})}
\end{equation}%
defines such a function $\alpha$, because $\alpha\alpha'=(1+\tanh)/2$. We explain the purpose of $\alpha$ later in \Cref{rmk:02}. Next, we choose a measurable and bounded $\mathbb{R}^n$-valued map $u$ on $\mathbb{R}$, which needs to be $T$-periodic and zero-mean for some positive period $T$. The component functions $u^1,\ldots,u^n$ of $u$ will act as dither signals in the control law. Let $U$ denote the zero-mean antiderivative of $u$. We demand that the dither signals are chosen in such a way that the component functions $U^1,\ldots,U^n$ of $U$ satisfy the orthonormality condition%
\begin{equation}\label{eq:07}
\frac{1}{T}\int_0^TU^i(\tau)\,U^j(\tau)\,\mathrm{d}\tau \ = \ \left\{ \begin{tabular}{cl} $1$ & if $i=j$, \\ 0 & if $i\neq{j}$ \end{tabular}\right.
\end{equation}%
for all indices $i,j$. For instance, we can define the component functions of $u$ by%
\begin{equation}\label{eq:08}
u^i(\tau) \ := \ \sqrt{2}\,i\,\cos(i\,\tau)
\end{equation}%
to satisfy~\cref{eq:07} with common period $T:=2\pi$. Finally, we choose vectors $\lambda$ and $\mu$ in $\mathbb{R}^n$ with components $\lambda^1,\ldots,\lambda^n$ and $\mu^1,\ldots,\mu^n$, which are specified later in equations~\cref{eq:15,eq:16} and also in \Cref{rmk:03}. For each input channel $\mathrm{u}^i$ in~\cref{eq:04}, we propose the parameter- and time-dependent output feedback control law%
\begin{equation}\label{eq:09}
\mathrm{u}^i \ = \ \omega\,u^i(\omega{t})\,\lambda^i\,\alpha(\mathrm{y}-\eta) + \mu^i\,\alpha^2(\mathrm{y}-\eta),
\end{equation}%
where $\omega$ is a positive real control parameter and $\eta$ is the real-valued state of the high-pass filter%
\begin{equation}\label{eq:10}
\dot{\eta} \ = \ -h\,\eta + h\,\mathrm{y}
\end{equation}%
with filter output $\mathrm{y}-\eta$ and positive gain $h$ to remove a possible offset from the measured output signal $\mathrm{y}$. We will see in the next subsection that control law~\cref{eq:09} gives access to descent directions of the objective function $\psi$.

%---------------------------------------------------------------------------------------------------------------------------------------------------------
%                                                                     Section 4.2
%---------------------------------------------------------------------------------------------------------------------------------------------------------

\subsection{Change of variables and averaged system}\label{sec:04.2}
By applying~\cref{eq:09} to~\cref{eq:04}, we obtain the closed-loop system%
\begin{subequations}\label{eq:11}%
\begin{align}
\nabla_{\dot{q}}\dot{q} & \ = \ Y_0(q) - R(q)\dot{q} \label{eq:11:a} \allowdisplaybreaks \\
& \qquad + \sum_{i=1}^n\mu^i\,\alpha^2(\psi(q)-\eta)\,Y_i(q) \label{eq:11:b} \allowdisplaybreaks \\
& \qquad + \sum_{i=1}^n\omega\,u^i(\omega{t})\,\lambda^i\,\alpha(\psi(q)-\eta)\,Y_i(q), \label{eq:11:c} \allowdisplaybreaks \\
\dot{\eta} & \ = \ -h\,\eta + h\,\psi(q). \label{eq:11:d}
\end{align}%
\end{subequations}%
Note that~\cref{eq:11:a,eq:11:b,eq:11:c} is a second-order differential equation on $Q$. Let $TQ$ denote the tangent bundle of $Q$. If we introduce the velocity $v:=\dot{q}$ as an additional variable, then~\cref{eq:11} can be equivalently written as a first-order differential equation on the product manifold $M$ of $TQ$ and $\mathbb{R}$. For this first-order system on $M$, we consider the change of variables%
\begin{subequations}\label{eq:12}%
\begin{align}
\tilde{q} & \ = \ q, \qquad \tilde{\eta} \ = \ \eta, \label{eq:12:a} \allowdisplaybreaks \\
\tilde{v} & \ = \ v - \sum_{i=1}^nU^i(\omega{t})\,\lambda^i\,\alpha(\psi(q)-\eta)\,Y_i(q), \label{eq:12:b}
\end{align}%
\end{subequations}%
which shifts the velocity $v$ within the tangent space at $q$ while leaving the configuration $q$ and the filter state $\eta$ unchanged. Following the averaging analysis in~\cite{Bullo2002,BulloBook}, we will see that the solutions of~\cref{eq:11} in the variables \cref{eq:12} approximate the solutions of a system with so-called symmetric products of the vector fields in \cref{eq:11:c} on the right-hand side. To make this precise, fix an arbitrary high-pass filter state $\eta\in\mathbb{R}$ and, for every $i\in\{1,\ldots,n\}$, define a vector field $X_i$ on $Q$ by%
\begin{equation}\label{eq:13}
X_i(q) \ := \ \lambda^i\,\alpha(\psi(q)-\eta)\,Y_i(q).
\end{equation}%
For all $i,j\in\{1,\ldots,n\}$, the \emph{symmetric product} of $X_i$ and $X_j$ with respect to $\nabla$ is the vector field%
\begin{equation*}
\langle{X_i\colon\!{X_j}}\rangle \ := \ \nabla_{X_j}X_{i} + \nabla_{X_i}X_j
\end{equation*}%
on $Q$; see~\cite{BulloBook}. (Alternatively, one can express the symmetric product of vector fields on $Q$ in terms of iterated Lie brackets of vector fields on $TQ$. However, we do not go into the mathematical details here, but instead we refer the interested reader to the textbook~\cite{BulloBook}.) Since the employed dither signals are assumed to satisfy the orthonormality condition \cref{eq:07}, our approach only leads to an approximation of symmetric products of each $X_i$ with itself. We compute%
\begin{align*}
-\tfrac{1}{2}\langle{X_i\colon\!{X_i}}\rangle(q) & \ = \ -\tfrac{1}{2}\alpha^2(\psi(q)-\eta)\,(\lambda^i)^2\,\langle{Y_i}\colon\!{Y_i}\rangle(q) \allowdisplaybreaks \\
& - (\alpha\alpha')(\psi(q)-\eta)\,(\lambda^i)^2\,(Y_i\psi)(q)\,Y_i(q)
\end{align*}%
for every $i\in\{1,\ldots,n\}$ and every $q\in{Q}$, where $(Y_i\psi)(q)$ denotes the Lie derivative of $\psi$ along $Y_i$ at $q$ (see~\cite{BulloBook}). If we replace the oscillatory terms in \cref{eq:11:c} by the above symmetric products, then we obtain the \emph{averaged system}%
\begin{subequations}\label{eq:14}%
\begin{align}
\nabla_{\dot{\bar{q}}}\dot{\bar{q}} & \ = \ Y_0(\bar{q}) - R(\bar{q})\dot{\bar{q}} \label{eq:14:a} \allowdisplaybreaks \\
& \qquad + \alpha^2(\psi(\bar{q})-\bar{\eta})\sum_{i=1}^n\mu^i\,Y_i(\bar{q}) \label{eq:14:b} \allowdisplaybreaks \\
& \qquad - \frac{1}{2}\,\alpha^2(\psi(\bar{q})-\bar{\eta})\sum_{i=1}^n(\lambda^i)^2\,\langle{Y_i}\colon\!{Y_i}\rangle(\bar{q}) \label{eq:14:c} \allowdisplaybreaks \\
& \qquad - (\alpha\alpha')(\psi(\bar{q})-\bar{\eta})\sum_{i=1}^n(\lambda^i)^2\,(Y_i\psi)(\bar{q})\,Y_i(\bar{q}), \label{eq:14:d} \allowdisplaybreaks \\
\dot{\bar{\eta}} & \ = \ -h\,\bar{\eta} + h\,\psi(\bar{q}). \label{eq:14:e}
\end{align}%
\end{subequations}%
\begin{remark}\label{rmk:02}
The expression in \cref{eq:14:d} is of particular interest for our objective to minimize the value of $\psi$. Note that if the Lie derivative of $\psi$ along $Y_i$ at some $q\in{Q}$ is nonzero, then the tangent vector $-(Y_i\psi)(q)\,Y_i(q)$ points into a descent direction of $\psi$ at $q$. This explains why we demand that $\alpha\alpha'$ only takes positive values: If at least one of the Lie derivatives $Y_i\psi$ is nonzero at $q\in{Q}$, then the vector in \cref{eq:14:d} points into a descent direction of $\psi$. The additional contribution in \cref{eq:14:c} from the symmetric product approximation can be seen as an undesired remainder. In certain situations, a suitable choice of the vector $\mu$ can lead to  the effect that the contribution in \cref{eq:14:b} eliminates the undesired remainder in \cref{eq:14:c}. We consider this case in \Cref{sec:04.3}
\end{remark}
As for the closed-loop system \cref{eq:11}, we can write \cref{eq:14} equivalently as a first-order differential equation on $TQ\times\mathbb{R}$. In the terminology of \Cref{def:04} with $M=TQ\times\mathbb{R}$, we can state the following approximation result.
\begin{proposition}\label{prop:02}
The solutions of \cref{eq:11} in the variables \cref{eq:12} approximate the solutions of \cref{eq:14}.
\end{proposition}
The proof of \Cref{prop:02} is similar to that of Theorem~9.32 in~\cite{BulloBook}. Because of space limitations, we omit the proof. (References to alternative treatments of averaging in mechanical systems subject to oscillatory controls, such as~\cite{Baillieul1995,Levi1999}, are given in~\cite{BulloBook}.) An immediate consequence of \Cref{prop:02,prop:01} is the following stability result in the terminology of \Cref{def:03,rmk:01} with $M=TQ\times\mathbb{R}$.
\begin{corollary}\label{cor:01}
For every nonempty and compact subset $K$ of $TQ\times\mathbb{R}$ and every neighborhood $S$ of $K$ in $TQ\times\mathbb{R}$, the following implication holds: If $K$ is $S$-uniformly asymptotically stable for \cref{eq:14}, then $K$ is $S$-practically uniformly asymptotically stable for \cref{eq:11} in the variables \cref{eq:12}.
\end{corollary}
In general, stability of the averaged system \cref{eq:14} is rather difficult to check. We present simple sufficient (but not necessary) conditions in the next subsection.

%---------------------------------------------------------------------------------------------------------------------------------------------------------
%                                                                     Section 4.3
%---------------------------------------------------------------------------------------------------------------------------------------------------------

\subsection{Stability conditions for the averaged system}\label{sec:04.3}
In this subsection, we additionally assume that $\nabla$ is the \emph{Levi-Civita connection} associated with a \emph{Riemannian metric} $\langle\!\langle\cdot,\cdot\rangle\!\rangle$ on $Q$ (see~\cite{BulloBook}). Then, the \emph{gradient of $\psi$}, denoted by $\operatorname{grad}\psi$, is the unique vector field on $Q$ such that $\langle\!\langle\operatorname{grad}\psi,Y\rangle\!\rangle=Y\psi$ for every vector field $Y$ on $Q$. We assume that the drift vector field $Y_0$ vanishes and that the components of the vectors $\lambda$ and $\mu$ in control law \cref{eq:09} can be chosen in such a way that%
\begin{align}
\sum_{i=1}^n(\lambda^i)^2\,(Y_i\psi)(q)\,Y_i(q) & \ = \ \kappa\,\operatorname{grad}\psi(q), \label{eq:15} \allowdisplaybreaks \\
\frac{1}{2}\sum_{i=1}^n(\lambda^i)^2\,\langle{Y_i\colon\!Y_i}\rangle(q) & \ = \ \sum_{i=1}^n\mu^i\,Y_i(q) \label{eq:16}
\end{align}%
at every point $q$ of $Q$, where $\kappa$ is a positive real constant (independent of $q$). Then the averaged system \cref{eq:14} reduces to%
\begin{subequations}\label{eq:17}%
\begin{align}
\nabla_{\dot{\bar{q}}}\dot{\bar{q}} & \ = \ - R(\bar{q})\dot{\bar{q}} - \kappa\,(\alpha\alpha')(\psi(\bar{q})-\bar{\eta})\,\operatorname{grad}\psi(\bar{q}), \label{eq:17:a} \allowdisplaybreaks \\
\dot{\bar{\eta}} & \ = \ -h\,\bar{\eta} + h\,\psi(\bar{q}). \label{eq:17:b}
\end{align}%
\end{subequations}%
In general, we cannot expect that there exist vectors $\lambda$ and $\mu$ such that \cref{eq:15,eq:16} hold. The subsequent remark describes a class of mechanical systems for which the conditions can be satisfied. Particular examples are given in \Cref{sec:05}.
\begin{remark}\label{rmk:03}
Suppose that $Q$ is a Lie group and that $\langle\!\langle\cdot,\cdot\rangle\!\rangle$ is left-invariant (see~\cite{BulloBook}). Suppose that the control vector fields $Y_1,\ldots,Y_n$ are left-invariant and that they form an orthogonal basis for the tangent space at some point $q_0$ of $Q$ with respect to $\langle\!\langle\cdot,\cdot\rangle\!\rangle_{q_0}$. Choose an arbitrary positive real number $\kappa$ and define the components of $\lambda$ by $\lambda^i:=\sqrt{\kappa}/\|Y_i(q_0)\|_{q_0}$. Then one can check that \cref{eq:15} is satisfied for every smooth real-valued function $\psi$ on $Q$ and every $q\in{Q}$. Since $\nabla$ is left-invariant, the expression on the left-hand side of \cref{eq:16} defines a left-invariant vector field on $Q$. Since the vector fields $Y_1,\ldots,Y_n$ form a basis for the space of left-invariant vector fields on $Q$, there exists a unique vector $\mu$ in $\mathbb{R}^n$ such that also \cref{eq:16} is satisfied at every point $q$ of $Q$.
\end{remark}
\begin{remark}\label{rmk:04}
The general stability implication in \Cref{cor:01} allows a general drift vector field $Y_0$: If the averaged system with drift is asymptotically stable, then we may conclude that the closed-loop system with drift is practically uniformly asymptotically stable. Sufficient conditions for asymptotic stability of the averaged system in the presence of a nonvanishing drift are left to future research. One can expect that the averaged system will still have certain stability properties if the impact of the drift vector field is sufficiently weak. On the other hand, if the destabilizing drift is stronger than the stabilizing gradient vector field, then, of course, we cannot expect stability. One possible attempt to overpower the drift is to choose sufficiently large gains $\lambda^i,\mu^i>0$ in control law~\cref{eq:09}. This amplifies the gradient vector field in the averaged system.
\end{remark}
To prove asymptotic stability for \cref{eq:17}, we introduce a suitable candidate Lyapunov function $V$ on $TQ\times\mathbb{R}$ by%
\begin{equation}\label{eq:18}
V(x) \ := \ \tfrac{1}{2}\,\|v\|_{q}^2 + \kappa\,(\alpha\alpha')(0)\,(\psi(q) \color{red} - y_\ast \color{black}) + \kappa\,\beta(\psi(q)-\eta)
\end{equation}%
for every $x=((q,v),\eta)\in{TQ}\times\mathbb{R}$, where $\beta\colon\mathbb{R}\to\mathbb{R}$ is defined by $\beta(z):=\int_0^z\big((\alpha\alpha')(\tilde{z})-(\alpha\alpha')(0)\big)\mathrm{d}\tilde{z}$. The first two terms on the right-hand side of \cref{eq:18} can be interpreted as the sum of kinetic and potential energy. The additional last term in \cref{eq:18} is due to the high-pass filter in \cref{eq:17:b}. The objective function $\psi$ is assumed to satisfy the following conditions.
\begin{assumption}\label{ass:01}
There exist $y_\ast,y_0\in\mathbb{R}$ with $y_\ast<y_0$ such that
\begin{enumerate}
	\item $y_\ast$ is the global minimum value of $\psi$;
	\item the $y_0$-sublevel set of $\psi$ is compact;
	\item $\operatorname{grad}\psi(q)\neq{0}$ for every $q\in{Q}$ with $y_\ast<\psi(q)\leq{y_0}$.
\end{enumerate}
\end{assumption}
As an abbreviation, for every $y\in\mathbb{R}$, we introduce the scaled sublevel set%
\begin{equation}\label{eq:19}
S(y) \ := \ \big\{x\in{M} \ \big| \ V(x) \ \leq \kappa(\alpha\alpha')(0)(y\color{red} - y_\ast \color{black})\big\}
\end{equation}%
of $V$. We say that the bundle map $-R$ in \cref{eq:17:a} is \emph{strictly dissipative} if $-\langle\!\langle{R(q)v},v\rangle\!\rangle_q<0$ for every $q\in{Q}$ and every $v\in{T_qQ}$ with $v\neq0$. Now we can state the following stability result for the averaged system in the terminology of \Cref{def:03} with $M=TQ\times\mathbb{R}$.
\begin{theorem}\label{thm:01}
Suppose that the bundle map $-R$ in \cref{eq:17} is strictly dissipative. Suppose that \Cref{ass:01} is satisfied with $y_\ast,y_0$ as therein. Then%
\begin{equation*}
S(y_\ast) \ = \ \big\{((q,v),\eta)\in{TQ}\times\mathbb{R} \ \big| \ \psi(q)=y_\ast=\eta, \ v=0\big\}
\end{equation*}%
is $S(y_0)$-uniformly asymptotically stable for \cref{eq:17}.
\end{theorem}
The proof of \Cref{thm:01} is similar to that of Theorem~6.47 in~\cite{BulloBook}. Because of space limitations, we omit the proof. Note that \Cref{thm:01,cor:01} imply practical asymptotic stability for the extremum seeking closed-loop system.

%---------------------------------------------------------------------------------------------------------------------------------------------------------
%---------------------------------------------------------------------------------------------------------------------------------------------------------
%                                                                     Section 5
%---------------------------------------------------------------------------------------------------------------------------------------------------------
%---------------------------------------------------------------------------------------------------------------------------------------------------------

\section{Examples}\label{sec:05}
In this section, we apply the theory from \Cref{sec:04} to two particular extremum seeking problems. Further applications can be found in~\cite{Suttner20192,Suttner20193}.

%---------------------------------------------------------------------------------------------------------------------------------------------------------
%                                                                     Section 5.1
%---------------------------------------------------------------------------------------------------------------------------------------------------------

\subsection{Planar rigid body with two thrusters}
In this first example, we study the problem of sources seeking with a thruster-driven planar rigid body. The system we consider here can be thought of as a model for a simplified hovercraft, as depicted in \Cref{fig:01}. Controllability properties of this model are investigated in~\cite{Lewis19972}. We refer to~\cite{BulloBook} for a detailed description of the system and its dynamics. The configuration manifold is $Q:=\mathbb{R}^2\times\mathbb{S}^1$, where $\mathbb{S}^1$ is the unit circle in $\mathbb{R}^2$. An element $q=(p,o)$ of $Q$ represents the current position $p\in\mathbb{R}^2$ of the center of mass and the current orientation $o\in\mathbb{S}^1$ of the body. A natural coordinate chart for the three-dimensional manifold $Q$ is of the form $(p^1,p^2,\theta)$, where $p^1,p^2$ are the components of the position vector and $\theta$ is an angle to describe the orientation. The rigid body is assumed to have mass $m>0$ and principal moment of inertia $J>0$. Two control forces can be applied at a point that is distance $\ell>0$ from the center of mass. One of these forces acts along a fixed direction while the other force acts along a variable direction as shown in \Cref{fig:01}. The current direction of the second force is described by an angle $\phi$. In natural coordinates $(p^1,p^2,\theta)$, the equations of motion are assumed to be of the form %
\begin{subequations}\label{eq:20}%
\begin{align}
m\,\ddot{p}^1 & \ = \ - r_1\,\dot{p}^1 + f^1\,\cos\theta + f^2\,\cos(\theta+\phi), \allowdisplaybreaks \\
m\,\ddot{p}^2 & \ = \ - r_2\,\dot{p}^2 + f^1\,\sin\theta + f^2\,\sin(\theta+\phi), \allowdisplaybreaks \\
J\,\ddot{\theta} & \ = \ - r_3\,\dot{\theta} - \ell\,f^2\,\sin\phi,
\end{align}%
\end{subequations}%
where $f^1,f^2$ are real-valued input channels for the control forces, and $r_1,r_2,r_3$ are non-negative real numbers to describe velocity-dependent dissipation.

We also assume that the planar body is equipped with a suitable sensor at its center of mass so that it can measure the value of a purely position-dependent signal. The signal is assumed to be given by an unknown smooth real-valued function $\psi$ on $\mathbb{R}^2$. A measurement of $\psi$ in the configuration $q=(p,o)\in{Q}$ results in the value%
\begin{equation}\label{eq:21}
\mathrm{y} \ = \ \psi(p).
\end{equation}%
We are interested in a feedback law that requires only real-time measurements of \cref{eq:21} and stabilizes the closed-loop system about positions where $\psi$ attains a minimum value. We propose the following control strategy. As in \Cref{sec:04.1}, let $\alpha$ be a smooth real-valued function on $\mathbb{R}$ such that $\alpha\alpha'$ is positive and strictly increasing. We apply%
\begin{equation*}
f^1 \, = \, \mu\,\alpha^2(\mathrm{y}-\eta), \quad f^2 \, = \, \sqrt{2}\,\lambda\,\omega\,\alpha(\mathrm{y}-\eta), \quad \phi \, = \, \omega{t}
\end{equation*}%
to \cref{eq:20}, where $\lambda,\mu$ are real constants which are be specified later in \cref{eq:24}, $\omega$ is a positive real control parameter, $\mathrm{y}$ is the measured signal \cref{eq:21}, and $\eta$ is the high-pass filter \cref{eq:10}. Using well-known trigonometric identities, we obtain the coordinate representation%
\begin{subequations}\label{eq:22}%
\begin{align}
\dot{p}^1 & \ = \ w^1, \qquad \dot{p}^2 \ = \ w^2, \qquad \dot{\theta} \ = \ \Omega, \allowdisplaybreaks \\
m\,\dot{w}^1 & \ = \ - r_1\,w_1 + \mu\,\alpha^2(\psi(p)-\eta)\,\cos\theta \allowdisplaybreaks \\
& \qquad + \omega\,u^1(\omega{t})\,\lambda\,\alpha(\psi(p)-\eta)\,\cos\theta \allowdisplaybreaks \\
& \qquad - \omega\,u^2(\omega{t})\,\lambda\,\alpha(\psi(p)-\eta)\,\sin\theta, \allowdisplaybreaks \\
m\,\dot{w}^2 & \ = \ - r_2\,w_2 + \mu\,\alpha^2(\psi(p)-\eta)\,\sin\theta \allowdisplaybreaks \\
& \qquad + \omega\,u^1(\omega{t})\,\lambda\,\alpha(\psi(p)-\eta)\,\sin\theta \allowdisplaybreaks \\
& \qquad + \omega\,u^2(\omega{t})\,\lambda\,\alpha(\psi(p)-\eta)\,\cos\theta, \allowdisplaybreaks \\
J\,\dot{\Omega} & \ = \ - r_3\,\Omega - \omega\,u^2(\omega{t})\,\ell\,\lambda\,\alpha(\psi(p)-\eta), \allowdisplaybreaks \\
\dot{\eta} & \ = \ -h\,\eta + h\,\psi(p)
\end{align}%
\end{subequations}%
for the closed-loop system on $TQ\times\mathbb{R}$, where the real-valued functions $u^1,u^2$ on $\mathbb{R}$ are defined by $u^1(\tau):=\sqrt{2}\cos(\tau)$, $u^2(\tau):=\sqrt{2}\sin(\tau)$. The zero-mean antiderivatives $U^1,U^2$ of $u^1,u^2$ are given by $U^1(\tau)=\sqrt{2}\sin(\tau)$, $U^2(\tau)=-\sqrt{2}\cos(\tau)$, which obviously satisfy \cref{eq:07} with $T:=2\pi$.
\begin{figure}%
\centering$\begin{matrix}\begin{matrix}\includegraphics{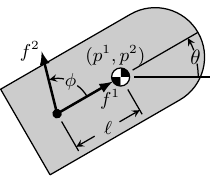}\\\vphantom{.}\end{matrix}&\begin{matrix}\includegraphics{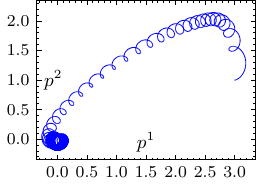}\end{matrix}\end{matrix}$%
\caption{Left: A planar rigid body with one direction-fixed thruster and one direction-variable thruster. Right: Simulation result for the position of a source-seeking planar rigid body, where the source is at the origin.}%
\label{fig:01}%
\end{figure}

Note that \cref{eq:22} is a particular case of the general closed-loop system \cref{eq:11}. To make this apparent, let $\frac{\partial}{\partial{p^1}},\frac{\partial}{\partial{p^2}},\frac{\partial}{\partial\theta}$ denote the coordinate vector fields of a natural chart $(p^1,p^2,\theta)$ for $Q$. Define a trivial connection $\nabla$ on $Q$ by asking that $\nabla_{\frac{\partial}{\partial{a}}}\frac{\partial}{\partial{b}}=0$ for all $a,b\in\{p^1,p^2,\theta\}$. Define a smooth bundle map $R$ over $Q$ by $R\frac{\partial}{\partial{p^i}}:=\frac{r_i}{m}\frac{\partial}{\partial{p^i}}$ for $i=1,2$ and $R\frac{\partial}{\partial\theta}:=\frac{r_3}{J}\frac{\partial}{\partial\theta}$. Define the vector fields $Y_1:=\hphantom{-}\frac{1}{m}\cos\theta\frac{\partial}{\partial{p^1}} + \frac{1}{m}\sin\theta\frac{\partial}{\partial{p^2}}$ and $Y_2:=-\frac{1}{m}\sin\theta\frac{\partial}{\partial{p^1}} + \frac{1}{m}\cos\theta\frac{\partial}{\partial{p^2}} - \frac{\ell}{J}\frac{\partial}{\partial\theta}$ on $Q$. By a slight abuse of notation, we denote the canonical extension of $\psi$ to a function on $Q$ by the same symbol. That is, we write $\psi(q)=\psi(p)$ for every $q=(p,o)\in{Q}$. Finally, set $\lambda^1:=\lambda^2:=\lambda$, $\mu^1:=\mu$, and $\mu^2:=0$. Now it is easy to see that the system with coordinate representation \cref{eq:22} is of the form \cref{eq:11} with vanishing drift $Y_0$ and $n:=2$ control vector fields.

We follow the general analysis in \Cref{sec:04.2}. One can check that the change of variables in \cref{eq:12} for the particular case of \cref{eq:22} has the coordinate representation%
\begin{subequations}\label{eq:23}%
\begin{align}
\tilde{p}^1 & \ = \ p^1, \qquad \tilde{p}^2 \ = \ p^2, \qquad \tilde{\theta} \ = \ \theta, \qquad \tilde{\eta} \ = \ \eta, \label{eq:23:a} \allowdisplaybreaks \\
\tilde{w}^1 & \ = \ w^1 - U^1(\omega{t})\,\tfrac{\lambda}{m}\,\alpha(\psi(p)-\eta)\,\cos\theta \label{eq:23:b} \allowdisplaybreaks \\
&\hphantom{{} \ = \ w^1 {}} + U^2(\omega{t})\,\tfrac{\lambda}{m}\,\alpha(\psi(p)-\eta)\,\sin\theta, \label{eq:23:c} \allowdisplaybreaks \\
\tilde{w}^2 & \ = \ w^2 - U^1(\omega{t})\,\tfrac{\lambda}{m}\,\alpha(\psi(p)-\eta)\,\sin\theta \label{eq:23:d} \allowdisplaybreaks \\
&\hphantom{{} \ = \ w^2 {}} - U^2(\omega{t})\,\tfrac{\lambda}{m}\,\alpha(\psi(p)-\eta)\,\cos\theta, \label{eq:23:e} \allowdisplaybreaks \\
\tilde{\Omega} & \ = \ \Omega - U^2(\omega{t})\,\tfrac{\ell\,\lambda}{J}\,\alpha(\psi(p)-\eta). \label{eq:23:f}
\end{align}%
\end{subequations}%
Next, we investigate the averaged system \cref{eq:14}. A direct computation shows that the symmetric products in \cref{eq:14:c} are given by $\langle{Y_1}\colon\!{Y_1}\rangle=0$, $\langle{Y_2}\colon\!{Y_2}\rangle=\tfrac{2\,\ell}{m\,J}\,Y_1$. This motivates us to select the constants%
\begin{equation}\label{eq:24}
\lambda \ := \ \sqrt{\kappa\,m}, \qquad \mu \ := \ \kappa\,\ell/J
\end{equation}%
with some $\kappa>0$. Now it is straightforward to check that the averaged system \cref{eq:14} is a cascade of a first-order system on $T\mathbb{R}^2\times\mathbb{R}$, represented by%
\begin{subequations}\label{eq:25}%
\begin{align}
\dot{\bar{p}}^1 & \ = \ \bar{w}^1, \qquad \dot{\bar{p}}^2 \ = \ \bar{w}^2, \allowdisplaybreaks \\
m\,\dot{\bar{w}}^1 & \ = \ - r_1\,\bar{w}_1 - \kappa\,(\alpha\alpha')(\psi(\bar{p})-\bar{\eta})\,\tfrac{\partial\psi}{\partial{p^1}}(p), \allowdisplaybreaks \\
m\,\dot{\bar{w}}^2 & \ = \ - r_2\,\bar{w}_2 - \kappa\,(\alpha\alpha')(\psi(\bar{p})-\bar{\eta})\,\tfrac{\partial\psi}{\partial{p^2}}(p), \allowdisplaybreaks \\
\dot{\bar{\eta}} & \ = \ -h\,\bar{\eta} + h\,\psi(\bar{p}),
\end{align}%
\end{subequations}%
followed by a first-order system on $T\mathbb{S}^1$, represented by%
\begin{subequations}\label{eq:26}%
\begin{align}
\dot{\bar{\theta}} & \ = \ \bar{\Omega}, \allowdisplaybreaks \\
J\,\dot{\bar{\Omega}} & \ = \ - r_3\,\bar{\Omega} + \mu\,\upsilon_1(\bar{p},\bar{\eta})\sin\theta - \mu\,\upsilon_2(\bar{p},\bar{\eta})\,\cos\theta,
\end{align}%
\end{subequations}%
where the system state of \cref{eq:25} enters through the ``inputs''%
\begin{equation}\label{eq:27}
\upsilon_i(\bar{p},\bar{\eta}) \ := \ (\alpha\alpha')(\psi(\bar{p})-\bar{\eta})\,\tfrac{\partial\psi}{\partial{p^i}}(\bar{p}), \qquad i=1,2.
\end{equation}%
To present a semi-global stability result for \cref{eq:22}, we need the following globalized version of \Cref{ass:01}.
\begin{assumption}\label{ass:02}
There exists $y_\ast\in\mathbb{R}$ such that
\begin{enumerate}
	\item $y_\ast$ is the global minimum value of $\psi$;
	\item $y_\ast<\hat{y}:=\sup\{\psi(p)\,|\,p\in\mathbb{R}^2\}\in\mathbb{R}\cup\{+\infty\}$;
	\item for every $y_0<\hat{y}$, the $y_0$-sublevel set of $\psi$ is compact;
	\item at every $p\in\mathbb{R}^2$ with $\psi(p)>y_\ast$, the derivative of $\psi$ is nonzero.
\end{enumerate}
\end{assumption}
Now we can state the following stability result for \cref{eq:22} in the terminology of \Cref{def:03,rmk:01} with $M=TQ\times\mathbb{R}$.
\begin{theorem}\label{thm:02}
Define the constants $\lambda,\mu$ in \cref{eq:22} by \cref{eq:24}. Suppose that the constants $r_1,r_2,r_3$ in \cref{eq:22} are positive. Suppose that \Cref{ass:02} is satisfied with $y_\ast$ as therein. Then the set%
\begin{equation*}
K := \big\{(((p,o),v),\eta)\in{TQ}\times\mathbb{R} \ \big| \ \psi(p)=y_\ast=\eta, \ v=0\big\}
\end{equation*}%
is semi-globally practically uniformly asymptotically stable for the system with coordinate representation \cref{eq:22} in the variables~\cref{eq:23}.
\end{theorem}
\begin{proof}
Note that \cref{eq:25} represents a system of the form \cref{eq:17}. Because of the global \Cref{ass:02}, one can use \Cref{thm:01} to prove that the compact subset%
\begin{equation*}
P \ := \ \big\{((\bar{p},\bar{w}),\bar{\eta})\in{T\mathbb{R}^2}\times\mathbb{R} \ \big| \ \psi(p)=y_\ast=\eta, \ w=0\big\}
\end{equation*}%
of ${T\mathbb{R}^2}\times\mathbb{R}$ is globally asymptotically stable for \cref{eq:25}. Using that $r_3$ is positive, one can also show that the compact subset $\Theta:=\mathbb{S}^1\times\{0\}$ of $T\mathbb{S}^1$ is input-to-state stable for \cref{eq:26} with \cref{eq:27} as ``inputs''. We refer to~\cite{Sontag1996} for a definition of input-to-state stability with respect to compact sets. The ``inputs'' \cref{eq:27} of system \cref{eq:26} vanish if the system state of \cref{eq:25} is an element of $P$. This implies that the set $P\times\Theta$ is globally uniformly asymptotically stable for the cascade of \cref{eq:25} and \cref{eq:26}. In other words, the set $K$ in \Cref{thm:02} is globally asymptotically stable for the averaged system of \cref{eq:22} in the variables \cref{eq:23}. Therefore the claim follows from \Cref{cor:01}.
\end{proof}
We end this examples by providing a numerical simulation result. For this purpose, we generate the signal measurements in \cref{eq:21} by $\psi(p):=-4\exp(-|p|^2/4)$ for every $p\in\mathbb{R}^2$. This defines a function $\psi$ which satisfies \Cref{ass:02}. The function $\alpha$ is defined by \cref{eq:06}. We set the constants $m,\kappa,\ell,h,r_1,r_2,r_3$ equal to $1$. The constants $\lambda,\mu$ are defined by \cref{eq:24}. Then, all assumptions of \Cref{thm:02} are satisfied, and therefore the set of desired states is semi-globally practically uniformly asymptotically stable for the system with coordinate representation \cref{eq:22} in the variables \cref{eq:23}. We generate numerical data for initial time $t_0:=0$, initial position $p_0:=(3,1)$, initial orientation $o_0:=(1,0)$, zero initial velocity, and initial filter state $\eta_0:=0$. It turns out that the parameter $\omega:=10$ is sufficiently large to ensure converge of the system state into a neighborhood of the set $K$ in \Cref{thm:02}. We can see in the plot on the right of \Cref{fig:01} that the second component $t\mapsto{p^2(t)}$ of the position $t\mapsto{p(t)}$ first increases from $p^2(0)=1$ to values $\approx2$ and then converges into a neighborhood of the origin. This behavior can be explained by the change of variables \cref{eq:23}. If we apply \cref{eq:23:b}-\cref{eq:23:e} to the initial velocity components $w^1=0$ and $w^2=0$ at time $t=0$, then we get $\tilde{w}^1=0$ and $\tilde{w}^2=\sqrt{2\kappa/m}\alpha(\psi(p_0)-\eta_0)$. Thus, by \Cref{prop:02} and \Cref{def:04}, the closed-loop system approximates the behavior of an averaged system with initial velocity components $\tilde{w}^1=0$ and $\tilde{w}^2>0$. This explains the initial increase of $t\mapsto{p^2(t)}$ in \Cref{fig:01}.

%---------------------------------------------------------------------------------------------------------------------------------------------------------
%                                                                     Section 5.2
%---------------------------------------------------------------------------------------------------------------------------------------------------------

\subsection{Fully actuated rigid body fixed at a point}
As a second illustrative example of the results in \Cref{sec:04}, we consider a problem of attitude control for a fully actuated rigid body fixed at a point. We refer to~\cite{BulloBook} for an introduction to rigid body dynamics. The model we consider here is often used to describe the attitude dynamics of satellites. The configuration manifold $Q$ is the Lie group $\operatorname{SO}(3)$ of real $3\times{3}$ rotation matrices. It is known that the tangent space to $\operatorname{SO}(3)$ at the identity is equal to the set $\mathfrak{so}(3)$ of real skew-symmetric $3\times{3}$ matrices. Since we deal with a matrix Lie group, the tangent space at any $q\in\operatorname{SO}(3)$ is given by $q\cdot\mathfrak{so}(3)$, where ``$\cdot$'' is the usual matrix multiplication. The elements of $\mathfrak{so}(3)$ are referred to as \emph{body angular velocities}. It is easy to check that the map%
\begin{equation*}
\mathbb{R}^3\ni\Omega=\left[\begin{smallmatrix} \Omega^1 \\ \Omega^2 \\ \Omega^3 \end{smallmatrix}\right] \mapsto \hat{\Omega}:=\left[\begin{smallmatrix} 0 & -\Omega^3 & \hphantom{-}\Omega^2 \\ \hphantom{-}\Omega^3 & 0 & -\Omega^1 \\ -\Omega^2 & \hphantom{-}\Omega^1 & 0 \end{smallmatrix}\right]\in\mathfrak{so}(3)
\end{equation*}%
is a vector space isomorphism $\hat{\cdot}\colon\mathbb{R}^3\to\mathfrak{so}(3)$. Its inverse is denoted by $\check{\cdot}\colon\mathfrak{so}(3)\to\mathbb{R}^3$. Thus, every body angular velocity $v\in\mathfrak{so}(3)$ has a unique vector representation $\check{v}\in\mathbb{R}^3$. Let $e_1,e_2,e_3$ denote the standard basis for $\mathbb{R}^3$, and let $J_1,J_2,J_3$ be positive real constants. Let $\mathbb{I}$ be the inner product on $\mathfrak{so}(3)$ whose representation matrix $[\mathbb{I}]\in\mathbb{R}^{3\times{3}}$ with respect to the basis $\hat{e}_1,\hat{e}_2,\hat{e}_3$ is diagonal with diagonal elements $J_1,J_2,J_3$. From a physical point of view, a diagonal representation matrix $[\mathbb{I}]$ means that we assume that the principal axes of the rigid body are aligned with the axes of the body references frame. The constants $J_1,J_2,J_3$ are the principal moments of inertia. For all $\Omega,\tilde{\Omega}\in\mathbb{R}^3$, let $\Omega\times\tilde{\Omega}\in\mathbb{R}^3$ denote the usual cross product of $\Omega,\tilde{\Omega}$. Let $r$ be a smooth map from $\operatorname{SO}(3)$ to $\mathbb{R}^{3\times{3}}$ to describe dissipation. Using the vector representation of body angular velocities, we assume that the rotation of the rigid body can be described by the \emph{controlled Euler equations}%
\begin{subequations}\label{eq:28}%
\begin{align}
\dot{q} & \ = \ q\cdot\hat{\Omega}, \\
[\mathbb{I}]\dot{\Omega} + \Omega\times([\mathbb{I}]\Omega) & \ = \ -r(q)\Omega + \mathrm{u}
\end{align}%
\end{subequations}%
on $\operatorname{SO}(3)\times\mathbb{R}^3$, where $\mathrm{u}$ is a row vector of real-valued input channels $\mathrm{u}^1,\mathrm{u}^2,\mathrm{u}^3$ for a control law (see~\cite{BulloBook}). The inputs $\mathrm{u}^1,\mathrm{u}^2,\mathrm{u}^3$ determine the torques about the principal axes as shown in the sketch on the left of \Cref{fig:02}.

We assume that the position-fixed rigid body can measure the values of a smooth orientation dependent function $\psi\colon\operatorname{SO}(3)\to\mathbb{R}$. To stabilize the body around orientations where $\psi$ attains a minimum value, we use the method from \Cref{sec:04}. Note that control system \cref{eq:28} represents a particular case of the general control system \cref{eq:04} in \Cref{sec:04}. To make this apparent, we let $\langle\!\langle\cdot,\cdot\rangle\!\rangle$ denote the left-invariant Riemannian metric on $\operatorname{SO}(3)$ that originates from $\mathbb{I}$ via left translations. Let $\nabla$ be the left-invariant Levi-Civita connection on $\operatorname{SO}(3)$ induced by $\langle\!\langle\cdot,\cdot\rangle\!\rangle$. Define a smooth bundle map $R$ over $\operatorname{SO}(3)$ by $R(q)v:=q\cdot\big([\mathbb{I}]^{-1}{r(q)}\big(q^{-1}\cdot{v}\big)\check{\vphantom{\cdot}}\big)\hat{\vphantom{\cdot}}$ for every element $(q,v)$ of the tangent bundle of $\operatorname{SO}(3)$, where $[\mathbb{I}]^{-1}$ and $q^{-1}$ denote the inverse matrices of $[\mathbb{I}]$ and $q$, respectively. For every $i\in\{1,2,3\}$, define a smooth left-invariant vector field $Y_i$ on $\operatorname{SO}(3)$ by $Y_i(q):=\frac{1}{J_i}\,q\cdot\hat{e}_i$. Now it is easy to verify that the controlled Euler equations \cref{eq:28} on $\operatorname{SO}(3)\times\mathbb{R}^3$ represent a system of the form \cref{eq:04} on $TQ$ with vanishing drift $Y_0$ and $n:=3$ control vector fields. Thus, we can apply control law \cref{eq:09} to \cref{eq:28}, which leads to a closed-loop system on $TQ\times\mathbb{R}$ with representation%
\begin{subequations}\label{eq:29}%
\begin{align}
\dot{q} & \ = \ q\cdot\hat{\Omega}, \label{eq:29:a} \allowdisplaybreaks \\
[\mathbb{I}]\dot{\Omega} & \ = \ -(\Omega\times([\mathbb{I}]\Omega)) - r(q)\Omega \label{eq:29:b} \allowdisplaybreaks \\
& \qquad + \sum_{i=1}^n\mu^i\,\alpha^2(\psi(q)-\eta)\,e_i \label{eq:29:c} \allowdisplaybreaks \\
& \qquad + \sum_{i=1}^n\omega\,\lambda^i\,u^i(\omega{t})\,\alpha(\psi(q)-\eta)\,e_i, \label{eq:29:e} \allowdisplaybreaks \\
\dot{\eta} & \ = \ -h\,\eta + h\,\psi(q) \label{eq:29:r}
\end{align}%
\end{subequations}%
on $\operatorname{SO}(3)\times\mathbb{R}^3\times\mathbb{R}$. In \cref{eq:29}, the functions $\alpha,u^i$ are chosen as in \Cref{sec:04.1}, and the constants $\lambda^i,\mu^i$ are specified later in \cref{eq:31}. The change of variables in \cref{eq:12} is represented by%
\begin{subequations}\label{eq:30}%
\begin{align}
\tilde{q} & \ = \ q, \qquad \tilde{\eta} \ = \ \eta, \label{eq:30:a} \allowdisplaybreaks \\
\tilde{\Omega} & \ = \ \Omega - \sum_{i=1}^3U^i(\omega{t})\,\frac{\lambda^i}{J_i}\,\alpha(\psi(q)-\eta)\,e_i. \label{eq:30:b}
\end{align}%
\end{subequations}%
Note that we are in a particular situation of \Cref{rmk:03}. A direct computation shows that $\|Y_i\|=1/\sqrt{J_i}$ and $\langle{Y_i}\colon\!Y_i\rangle=0$ for every $i\in\{1,2,3\}$. Following the recipe in \Cref{rmk:03}, we set%
\begin{equation}\label{eq:31}
\lambda^i \ := \ \sqrt{J_i\kappa} \qquad \text{and} \qquad \mu^i \ := \ 0
\end{equation}%
for every $i\in\{1,2,3\}$ with some constant $\kappa>0$. Then, the averaged system reduces to \cref{eq:17}. The candidate Lyapunov function $V$ on $TQ\times\mathbb{R}$ in \cref{eq:18} is represented by%
\begin{equation*}
V(q,\Omega,\eta) \, = \, \tfrac{1}{2}\,\Omega^{\top}[\mathbb{I}]\Omega + \kappa\,(\alpha\alpha')(0)\,(\psi(q)\color{red} - y_\ast \color{black}) + \kappa\,\beta(\psi(q)-\eta)
\end{equation*}%
on $\operatorname{SO}(3)\times\mathbb{R}^3\times\mathbb{R}$, where $\Omega^\top$ denotes the transpose of the vector $\Omega$. As in \cref{eq:19}, for every $y\in\mathbb{R}$, let $S(y)$ denote the $\kappa(\alpha\alpha')(0)(y\color{red} - y_\ast \color{black})$-sublevel set of the above representation of $V$. Finally, we say that the map $-r\colon\operatorname{SO}(3)\to\mathbb{R}^{3\times{3}}$ in \cref{eq:29:b} is \emph{strictly dissipative} if $-\Omega^\top{r(q)}\Omega<0$ for every $q\in\operatorname{SO}(3)$ and every $\Omega\in\mathbb{R}^3$ with $\Omega\neq0$. Now we can state the following consequence of \Cref{thm:01,cor:01}.
\begin{theorem}\label{thm:03}
For every $i\in\{1,2,3\}$, define the constants $\lambda^i,\mu^i$ in \cref{eq:29} by \cref{eq:31}. Suppose that the map $-r$ in \cref{eq:29} is strictly dissipative. Suppose that \Cref{ass:01} is satisfied with $y_\ast,y_0$ as therein. Then the set%
\begin{equation*}
\big\{(q,\Omega,\eta)\in\operatorname{SO}(3)\times\mathbb{R}^3\times\mathbb{R} \ \big| \ \psi(q)=y_\ast=\eta,\ \Omega=0\big\}
\end{equation*}%
is $S(y_0)$-practically uniformly asymptotically stable for \cref{eq:29} in the variables \cref{eq:30}.
\end{theorem}
Finally, we test the method numerically for the following choice of functions and parameters. We define the function $\alpha$ and the dither signals $u^i$ by \cref{eq:06} and \cref{eq:08}, respectively. For every square matrix $A$, let $\operatorname{tr}(A)$ denote the trace of $A$ and let $A^\top$ denote the transpose of $A$. We generate the signal measurements $\mathrm{y}=\psi(q)$ in every configuration $q\in\operatorname{SO}(3)$ by $\psi(q):=\operatorname{tr}\big((q-I_3)^\top(q-I_3)\big)/2$, where $I_3\in\operatorname{SO}(3)$ denotes the identity matrix. One can show that \Cref{ass:01} is satisfied for $y_\ast:=0$ and every $y_0\in(0,4)$. We define the map $r$ by $r(q):=[\mathbb{I}]/5$ for every $q\in\operatorname{SO}(3)$. For every $i\in\{1,2,3\}$, we set $J_i:=i$ and we define $\lambda^i,\mu^i$ by \cref{eq:31} with $\kappa:=1/10$. Finally, we set $h:=1$. Then, all assumptions of \Cref{thm:03} are satisfied, and therefore $(I_3,0,0)\in\operatorname{SO}(3)\times\mathbb{R}^3\times\mathbb{R}$ is at least locally practically uniformly asymptotically stable for \cref{eq:29} in the variables \cref{eq:30}. We generate numerical data for initial time $t_0:=0$, initial configuration%
\begin{equation*}
q_0 \ := \ \tfrac{1}{3}\left[\begin{smallmatrix}-1 & \hphantom{-}2 & -2 \\ -2 & \hphantom{-}1 & \hphantom{-}2 \\ \hphantom{-}2 & \hphantom{-}2 & 1 \end{smallmatrix}\right] \in \operatorname{SO}(3),
\end{equation*}%
zero initial velocity, and initial filter state $\eta_0:=0$. For $\omega:=10$ we can observe in the plot on the right of \Cref{fig:02} that the configuration $q(t)$ converges into a neighborhood of $I_3$ with increasing time parameter $t$.
\begin{figure}%
\centering$\begin{matrix}\begin{matrix}\includegraphics{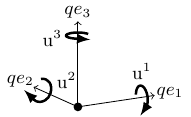}\end{matrix} \ & \ \begin{matrix}\includegraphics{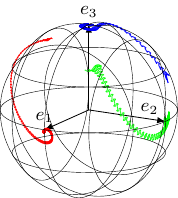}\end{matrix}\end{matrix}$%
\caption{Left: Rigid body fixed at the origin. In every configuration $q\in\operatorname{SO}(3)$, the principal axes of the body are given by the unit vectors $qe_1,qe_2,qe_3\in\mathbb{R}^3$. The control torques $\mathrm{u}^1,\mathrm{u}^2,\mathrm{u}^3$ act about the principal axes. Right: Simulation result for the maps $t\mapsto{q(t)e_1}$ (red), $t\mapsto{q(t)e_2}$ (green), $t\mapsto{q(t)e_3}$ (blue) on the unit sphere in $\mathbb{R}^3$.}%
\label{fig:02}%
\end{figure}

%---------------------------------------------------------------------------------------------------------------------------------------------------------
%---------------------------------------------------------------------------------------------------------------------------------------------------------
%                                                                     Section 6
%---------------------------------------------------------------------------------------------------------------------------------------------------------
%---------------------------------------------------------------------------------------------------------------------------------------------------------

\section{Conclusions}\label{sec:06}
We have proposed a novel extremum seeking method for affine connection mechanical control systems, which is based on approximations of symmetric products of vector fields. This differential geometric approach for second-order mechanical systems can be seen as a counterpart to existing Lie bracket approximations for first-order kinematic control systems. Our theoretical analysis has shown that practical asymptotic stability for the closed-loop system can be concluded from asymptotic stability for the associated averaged system. We have also given simple sufficient conditions to ensure that the averaged system is asymptotically stable. Our findings can be used to design extremum seeking control for many different problem, which was also illustrated by examples.

\bibliographystyle{abbrv}
\bibliography{bibFile}

\end{document}